\documentclass[pdflatex,sn-mathphys-num]{sn-jnl}

\usepackage{graphicx}
\usepackage{multirow}
\usepackage{amsmath,amssymb,amsfonts}
\usepackage{amsthm}
\usepackage{array}
\usepackage{mathrsfs}
\usepackage[title]{appendix}
\usepackage{xcolor}
\usepackage{textcomp}
\usepackage{manyfoot}
\usepackage{booktabs}
\usepackage{algorithm}
\usepackage{algorithmicx}
\usepackage{algpseudocode}
\usepackage{listings}
\usepackage{hyperref}
\usepackage[nameinlink,noabbrev]{cleveref}

 \theoremstyle{thmstyleone}
  \newtheorem{theorem}{Theorem}
  \newtheorem{proposition}{Proposition}
  \newtheorem{lemma}{Lemma}
  \newtheorem{corollary}{Corollary}
  \theoremstyle{thmstyletwo}
  \newtheorem{remark}{Remark}
  \theoremstyle{thmstylethree}
  \newtheorem{definition}{Definition}

\crefname{equation}{Eq.}{Eqs.}
\crefname{figure}{Fig.}{Figs.}
\crefname{table}{Table}{Tables}
\crefname{theorem}{Theorem}{Theorems}
\crefname{proposition}{Proposition}{Propositions}
\crefname{lemma}{Lemma}{Lemmas}
\crefname{corollary}{Corollary}{Corollaries}
\crefname{section}{Section}{Sections}
\crefname{appendix}{Appendix}{Appendices}

\DeclareMathOperator*{\argmax}{arg\,max}

\DeclareMathOperator{\supp}{supp}
\DeclareMathOperator{\Cov}{Cov}

\newcommand{\dd}{\,\mathrm{d}}
\newcommand{\R}{\mathbb{R}}
\newcommand{\E}{\mathbb{E}}
\newcommand{\calX}{\mathcal{X}}
\newcommand{\calM}{\mathcal{M}}
\newcommand{\calP}{\mathcal{P}}
\newcommand{\calC}{\mathcal{C}}
\newcommand{\calS}{\mathcal{S}}
\newcommand{\calH}{\mathcal{H}}
\newcommand{\calJ}{\mathcal{J}}
\newcommand{\norm}[1]{\left\lVert #1\right\rVert}
\newcommand{\pos}[1]{\left[#1\right]_+}

\begin{document}

\title[Projective maximum entropy]{Projective Maximum Entropy: Universality and Acceptance-Region Calibration}

\author{\fnm{Hideitsu} \sur{Hino}}\email{hino@ism.ac.jp}

\affil{%
  \orgname{The Institute of Statistical Mathematics},
  \orgaddress{%
    \street{10-3 Midori cho},
    \city{Tachikawa City},
    \postcode{190-8562},
    \state{Tokyo},
    \country{Japan}}}

\abstract{
Maximum-entropy reference distributions are usually constructed on the normalized probability simplex. This formulation is less natural for unnormalized statistical models, in which positive multiples represent the same shape, and it does not directly explain how a prescribed admissible region should determine the deformation parameter of a bounded-support reference distribution. We formulate maximum entropy on the projective space of nonnegative measures and establish three results of statistical relevance. First, a universality theorem shows that every admissible monotone transform of the same normalized power functional has exactly the same optimizer under linear moment constraints. The result unifies the maximum-entropy implications of Tsallis and R\'enyi entropies, H\"older composite scores, pseudo-spherical scores, Bregman--H\"older constructions, and related homogeneous divergences without asserting a new distribution family. Second, the common optimizer is characterized as a $q$-exponential density; under mean and covariance constraints it is a compactly supported $q$-Gaussian for positive deformation and a Student-type density for negative deformation. Third, a prescribed Mahalanobis acceptance region with squared radius $R^2>d+2$ uniquely determines the deformation parameter $\gamma_R=2/(R^2-d-2)$. The resulting affine-equivariant reference density is the unique projective maximum-entropy solution, and its support coincides with the specified ellipsoid without an additional support constraint. This provides a principled method for constructing bounded-support statistical reference distributions from robust location and scatter estimates or from externally specified admissible regions.}

\keywords{maximum entropy; projective statistical model; H\"older score; $q$-exponential distribution; affine equivariance; acceptance region}

\maketitle

\section{Introduction}\label{sec:introduction}

The maximum-entropy principle provides a constructive rule for selecting a probability distribution from partial information and has long been interpreted as a method of statistical inference~\citep{Jaynes1957}. Under prescribed mean and covariance, the Shannon-entropy maximizer is Gaussian. This classical characterization provides a statistical rationale for using the Gaussian law as a reference distribution when only first- and second-order moment information is available. In a number of applications, however, a statistical reference distribution is also expected to respect an admissible region. Examples include safety envelopes, physically reachable states, feasible design domains, and acceptance regions used for anomaly screening. A Gaussian reference assigns positive probability outside every bounded set, whereas imposing a hard truncation after fitting generally destroys the prescribed covariance and lacks a variational justification.

A second motivation comes from unnormalized statistical models. An energy-based model is written as
\[
 r_\theta(x)=\exp\{-E_\theta(x)\},
 \qquad
 p_\theta(x)=\frac{r_\theta(x)}{\int r_\theta\,\dd\mu_0}.
\]
Adding a constant to the energy multiplies $r_\theta$ by a positive scalar. Hence the statistically relevant object for shape is the ray
\[
 [r_\theta]=\{c r_\theta:c>0\},
\]
not a particular unnormalized representative. This observation suggests formulating maximum entropy on a projective positive cone and treating normalization as a choice of representative rather than as part of the intrinsic object.

Covariance- and moment-constrained maximization of R\'enyi and related power entropies has been studied extensively. Generalized Gaussian densities are extremizers of sharp moment--entropy inequalities~\citep{LutwakYangZhang2005}, and the compactly supported Student-$r$ and heavy-tailed Student-$t$ branches maximize R\'enyi entropy under covariance constraints in their respective parameter ranges~\citep{JohnsonVignat2007}. More general bounds involving multiple moments have also been developed~\citep{Reeves2020}. The associated families have been studied geometrically through $q$-exponential models, R\'enyi geometry, projectively flat divergences, and generalized duality~\citep{AmariOhara2011,Wong2018,WongZhang2022,doi:10.1142/9789811296710_0013}. These results make clear that the distributional form derived below is not new. The question addressed here is instead whether a broad class of projective entropy functionals has the same optimizer and whether that common optimizer can be calibrated from a statistically specified acceptance region.

Proper scoring rules and their induced entropy and divergence functionals provide a complementary line of work~\citep{GneitingRaftery2007,Ovcharov2018}. Density-power, pseudo-spherical, logarithmic gamma, and H\"older-type scores share a common power structure and have distinct roles in robust estimation, affine-equivariant inference, and unnormalized modelling~\citep{BasuHarrisHjortJones1998,FujisawaEguchi2008,KanamoriFujisawa2014Bernoulli,Kanamori2014Entropy,HinoEguchi2023}. A related information-geometric phenomenon is entropy gauge freedom: distinct escort or divergence constructions can induce the same entropy while retaining different relative entropies~\citep{MatsuzoeTakatsu2021}. This distinction motivates separating the diagonal ordering that determines a maximum-entropy optimizer from the off-diagonal and radial structure that determines estimation and scale identification.

The methodological problem considered in this paper is therefore not whether each generalized divergence gives another named distribution. It is to identify the largest useful equivalence class for projective maximum entropy, to establish the common optimizer with global uniqueness under moment constraints, and to turn the deformation parameter into an interpretable acceptance-region parameter.

This paper makes three principal contributions.

First, we define a normalized power functional on the projective space of nonnegative measures and prove a maximum-entropy universality theorem. Any entropy that is an admissible strictly monotone transform of this functional yields exactly the same optimizer under the same moment constraints. This separates the common variational ordering from the distinct estimating properties of individual scoring rules and divergences.

Second, we derive the universal optimizer under linear moment constraints and establish global uniqueness in the mean--covariance problem. The optimizer is a $q$-exponential density. For positive deformation it is compactly supported; for negative deformation in the finite-covariance range it is a Student-type density. The contribution is not the introduction of a new $q$-Gaussian family, but a projective and statistically interpretable characterization that applies simultaneously to a broad class of entropy constructions.

Third, we formulate an acceptance-region calibration theorem. Given mean $\mu$, covariance $V$, and a Mahalanobis acceptance ellipsoid of squared radius $R^2>d+2$, the deformation parameter is uniquely fixed as
\[
 \gamma_R=\frac{2}{R^2-d-2}.
\]
The resulting density is the unique affine-equivariant maximum-entropy reference distribution in the universality class, and the prescribed ellipsoid emerges as its support. This turns the deformation parameter from an abstract tuning constant into a directly interpretable statistical design parameter. In practice, $\mu$ and $V$ may be supplied by robust location and scatter estimators, while $R$ may be specified by domain knowledge, safety requirements, or an externally defined acceptance rule.

The remainder of the paper is organized as follows. \Cref{sec:projective} develops the projective formulation and proves the equivalence between optimization on the quotient space and on the normalized section. \Cref{sec:universality} identifies the common power structure and states the universality theorem. \Cref{sec:qexp} derives the $q$-exponential representation. \Cref{sec:mean-cov} solves the mean--covariance problem, and \Cref{sec:calibration} gives the acceptance-region calibration theorem. \Cref{sec:shape-scale} separates projective shape selection from radial scale identification. Numerical illustrations are given in \Cref{sec:illustrations}, followed by discussion and conclusions. Technical integral calculations and the full proof of the mean--covariance theorem are collected in \Cref{app:mean-cov-proof}.

\section{Projective formulation}\label{sec:projective}

\subsection{The projective positive cone}

Let $(\calX,\mathcal A,\mu_0)$ be a $\sigma$-finite measure space, and let
\[
 \gamma\in(-1,\infty)\setminus\{0\}.
\]
Define
\[
\calM_{+,\gamma}
=
\left\{
 r:\calX\to[0,\infty):
 0<\int r\,\dd\mu_0<\infty,
 \quad
 \int r^{1+\gamma}\,\dd\mu_0<\infty
\right\}.
\]
For $r,s\in\calM_{+,\gamma}$, write $r\sim s$ if $s=cr$ for some $c>0$, and define the projective space
\[
 \mathbb P\calM_{+,\gamma}=\calM_{+,\gamma}/\!\sim.
\]
Each projective class $[r]$ has the unique normalized representative
\[
 \bar r=\frac{r}{\int r\,\dd\mu_0}.
\]

\begin{definition}[Projective power functional]\label{def:projective-power}
For $[r]\in\mathbb P\calM_{+,\gamma}$, define
\begin{equation}
\calJ_\gamma([r])
=
\int \bar r^{1+\gamma}\,\dd\mu_0
=
\frac{\int r^{1+\gamma}\,\dd\mu_0}
{\left(\int r\,\dd\mu_0\right)^{1+\gamma}}.
\label{eq:projective-power}
\end{equation}
\end{definition}

The right-hand side of \cref{eq:projective-power} is unchanged under $r\mapsto cr$, so $\calJ_\gamma$ is well defined on the quotient space. The Shannon functional on the same projective space is
\[
 \calS_0([r])=-\int \bar r\log \bar r\,\dd\mu_0.
\]

\subsection{Moment constraints and the normalized section}

Let $T=(T_1,\ldots,T_m):\calX\to\R^m$ be measurable. For a projective class for which the required moments are finite, define
\[
 \E_{[r]}[T]
 =\int T\bar r\,\dd\mu_0
 =\frac{\int Tr\,\dd\mu_0}{\int r\,\dd\mu_0}.
\]
This quantity depends only on $[r]$. For a prescribed $\tau\in\R^m$, introduce the projective feasible set
\[
 \widetilde{\calC}_{\tau,\gamma}
 =
 \left\{
 [r]\in\mathbb P\calM_{+,\gamma}:
 \int \lvert T_j\rvert\bar r\,\dd\mu_0<\infty,
 \ \E_{[r]}[T]=\tau
 \right\}
\]
and the normalized feasible section
\begin{equation}
\calC_{\tau,\gamma}
=
\left\{
 p\in\calM_{+,\gamma}:
 \int p\,\dd\mu_0=1,
 \quad
 \int \lvert T_j\rvert p\,\dd\mu_0<\infty,
 \quad
 \int Tp\,\dd\mu_0=\tau
\right\}.
\label{eq:feasible-set}
\end{equation}

The next proposition makes precise the statement that optimization on the projective space is the same problem as optimization on the normalized section.

\begin{proposition}[Normalization-section equivalence]\label{prop:normalization-section}
The map
\[
 N:\widetilde{\calC}_{\tau,\gamma}\to\calC_{\tau,\gamma},
 \qquad
 N([r])=\frac{r}{\int r\,\dd\mu_0},
\]
is a bijection. Moreover,
\begin{equation}
 \calJ_\gamma([r])
 =\int N([r])^{1+\gamma}\,\dd\mu_0.
\label{eq:objective-section}
\end{equation}
Consequently, for every real-valued function $F$ on the range of $\calJ_\gamma$,
\begin{equation}
 \sup_{[r]\in\widetilde{\calC}_{\tau,\gamma}}
 F\!\left(\calJ_\gamma([r])\right)
 =
 \sup_{p\in\calC_{\tau,\gamma}}
 F\!\left(\int p^{1+\gamma}\,\dd\mu_0\right),
\label{eq:section-equivalence}
\end{equation}
and maximizers correspond one-to-one under $N$.
\end{proposition}

\begin{proof}
If $s=cr$ with $c>0$, then
\[
 \frac{s}{\int s\,\dd\mu_0}
 =\frac{cr}{c\int r\,\dd\mu_0}
 =\frac{r}{\int r\,\dd\mu_0},
\]
so $N$ is well defined. The normalized representative has unit mass and preserves every projective moment, hence belongs to $\calC_{\tau,\gamma}$. If $N([r])=N([s])$, then
\[
 s=\frac{\int s\,\dd\mu_0}{\int r\,\dd\mu_0}r,
\]
so $[r]=[s]$ and $N$ is injective. Conversely, for any $p\in\calC_{\tau,\gamma}$, one has $N([p])=p$, proving surjectivity. Equation \cref{eq:objective-section} follows from \cref{eq:projective-power}, and \cref{eq:section-equivalence} follows from the bijection.
\end{proof}

\begin{remark}
Passing to $\calC_{\tau,\gamma}$ does not restrict the projective feasible set. It is a gauge choice: every positive ray intersects the unit-mass section exactly once. The only discarded quantity is the overall mass of an unnormalized representative, which is already identified in the projective quotient.
\end{remark}

\subsection{Admissible projective entropies}

\begin{definition}[$\gamma$-admissible projective entropy]\label{def:admissible}
Let $F_\gamma:(0,\infty)\to\R$ be strictly monotone and satisfy
\[
 F_\gamma \text{ is }
 \begin{cases}
 \text{strictly decreasing},&\gamma>0,\\
 \text{strictly increasing},&-1<\gamma<0.
 \end{cases}
\]
The functional
\begin{equation}
 \calH_{\gamma,F}([r])
 =F_\gamma\!\left(\calJ_\gamma([r])\right)
\label{eq:admissible-entropy}
\end{equation}
is called a $\gamma$-admissible projective entropy.
\end{definition}

Representative choices are listed in \cref{tab:entropy-class}. The terminology refers to the variational ordering, not to equality of the corresponding divergences away from the diagonal.

\begin{table}[t]
\centering
\small
\caption{Representative entropy constructions in the same projective maximum-entropy class. Here $u=\calJ_\gamma([r])$.}
\label{tab:entropy-class}
\begin{tabular}{>{\raggedright\arraybackslash}p{0.30\linewidth}p{0.33\linewidth}p{0.29\linewidth}}
\toprule
Construction & $F_\gamma(u)$ & Comment \\
\midrule
Tsallis entropy \citep{Tsallis1988}
& $(1-u)/\gamma$
& Has the Shannon limit \\
R\'enyi entropy \citep{Renyi1961}
& $-(1/\gamma)\log u$
& Order-equivalent to Tsallis \\
Logarithmic gamma score \citep{FujisawaEguchi2008}
& $-[\gamma(1+\gamma)]^{-1}\log u$
& Positive multiple of R\'enyi \\
H\"older composite score, $\gamma>0$ \citep{KanamoriFujisawa2014Bernoulli}
& $-u$
& Independent of its generating function on the diagonal \\
Generalized H\"older score, $-1<\gamma<0$ \citep{Kanamori2014Entropy}
& $-u^{-1}$
& Reverse-H\"older branch \\
Bregman--H\"older potential \citep{KanamoriFujisawa2014Bernoulli,Kanamori2014Entropy}
& $-u^{\kappa/(1+\gamma)}$
& For the admissible convexity range \\
Dual homogeneous potential \citep{HinoEguchi2023}
& $-u^{\gamma/(1+\gamma)}$
& Same projective power order \\
\bottomrule
\end{tabular}
\end{table}

\section{Universality and the H\"older structure}\label{sec:universality}

\subsection{A common normalized H\"older affinity}

For $\gamma>0$ and probability densities $p,q$, define
\begin{equation}
 \rho_\gamma(p,q)
 =
 \frac{\int p q^\gamma\,\dd\mu_0}
 {\left(\int p^{1+\gamma}\,\dd\mu_0\right)^{1/(1+\gamma)}
  \left(\int q^{1+\gamma}\,\dd\mu_0\right)^{\gamma/(1+\gamma)}}.
\label{eq:holder-affinity}
\end{equation}
H\"older's inequality gives $0\leq\rho_\gamma(p,q)\leq1$, with equality at one if and only if $p=q$ almost everywhere. Several standard discrepancies are monotone or radially weighted functions of the same affinity. For example, up to conventional positive constants, one may write
\begin{align}
 D_\gamma^{\log}(p,q)
 &= -\frac{1}{\gamma}\log\rho_\gamma(p,q),
\label{eq:log-holder-div}\\
 D_\gamma^{\mathrm{ps}}(p,q)
 &= \left(\int p^{1+\gamma}\,\dd\mu_0\right)^{1/(1+\gamma)}
    \{1-\rho_\gamma(p,q)\},
\label{eq:ps-div}\\
 D_\gamma^{\mathrm{dh}}(p,q)
 &= \left(\int q^{1+\gamma}\,\dd\mu_0\right)^{\gamma/(1+\gamma)}
    \{1-\rho_\gamma(p,q)\}.
\label{eq:dual-hom-div}
\end{align}
The superscripts refer, respectively, to the logarithmic H\"older form~\citep{FujisawaEguchi2008}, the pseudo-spherical form~\citep{Good1971}, and the dual homogeneous form~\citep{HinoEguchi2023}. The three discrepancies differ in their radial prefactors and in their behavior under scaling of the arguments, but they share the same projective angular component. This distinction is important: common maximum-entropy behavior does not imply identical estimation or robustness properties.

\subsection{Diagonal entropy of H\"older composite scores}

For $\gamma>0$, a H\"older composite score can be written as
\begin{equation}
 S_{\gamma,\phi}(f,g)
 =
 \phi\!\left(
 \frac{\int f g^\gamma\,\dd\mu_0}
 {\int g^{1+\gamma}\,\dd\mu_0}
 \right)
 \int g^{1+\gamma}\,\dd\mu_0,
\label{eq:holder-score}
\end{equation}
where $\phi(1)=-1$ and $\phi(z)\geq-z^{1+\gamma}$ under the usual properness conditions \citep{KanamoriFujisawa2014Bernoulli}. Its diagonal value is
\begin{equation}
 S_{\gamma,\phi}(p,p)
 =-\int p^{1+\gamma}\,\dd\mu_0,
\label{eq:holder-diagonal}
\end{equation}
which is independent of $\phi$. Thus the generating function changes off-diagonal behavior but not the maximum-entropy ordering. For $-1<\gamma<0$, the scale-invariant extension based on the reverse H\"older inequality has diagonal ordering represented by $-\calJ_\gamma(p)^{-1}$ on its natural domain \citep{Kanamori2014Entropy}.

This separation between diagonal entropy and off-diagonal discrepancy is consistent with the general convex-analytic theory of proper scoring rules, in which the entropy is the diagonal value whereas the associated divergence additionally depends on a supporting functional \citep{Ovcharov2018}. A closely related gauge-freedom phenomenon occurs on $q$-Gaussian families: different escort constructions can yield the same entropy but different relative entropies \citep{MatsuzoeTakatsu2021}. The result below is different in scope. It establishes optimizer equivalence on the full moment-constrained feasible set, rather than only an equivalence of geometric descriptions on a fixed parametric family.

\begin{theorem}[Projective maximum-entropy universality]\label{thm:universality}
Let $\gamma\in(-1,\infty)\setminus\{0\}$ and let $\calC_{\tau,\gamma}$ be defined by \cref{eq:feasible-set}. For every $\gamma$-admissible projective entropy,
\begin{equation}
 \argmax_{p\in\calC_{\tau,\gamma}}
 \calH_{\gamma,F}(p)
 =
 \argmax_{p\in\calC_{\tau,\gamma}}
 \calS_\gamma(p),
 \qquad
 \calS_\gamma(p)=\frac{1-\int p^{1+\gamma}\,\dd\mu_0}{\gamma}.
\label{eq:universality-argmax}
\end{equation}
If $\calC_{\tau,\gamma}$ is convex and an optimizer exists, the optimizer is unique.
\end{theorem}

\begin{proof}
For $\gamma>0$, maximizing $\calS_\gamma$ is equivalent to minimizing $\int p^{1+\gamma}\,\dd\mu_0$. Every admissible $F_\gamma$ is strictly decreasing, so it induces the same minimization problem. For $-1<\gamma<0$, maximizing $\calS_\gamma$ is equivalent to maximizing the power integral, and admissible $F_\gamma$ is strictly increasing. Uniqueness follows because $t\mapsto t^{1+\gamma}$ is strictly convex for $\gamma>0$ and strictly concave for $-1<\gamma<0$.
\end{proof}

\begin{corollary}[Common optimizer for the H\"older class]\label{cor:holder-universality}
For fixed $\gamma$, the projective maximum-entropy problems induced by Tsallis and R\'enyi entropies, H\"older composite scores, logarithmic gamma scores, pseudo-spherical scores, Bregman--H\"older potentials, and dual homogeneous potentials have the same solution whenever they are defined on the same feasible set.
\end{corollary}

\subsection{Implication of affine-invariance characterizations}

Within a regular two-integral composite-score class of the form
\[
 S(f,g)=\psi\!\left(\int fU(g)\,\dd\mu_0,\int V(g)\,\dd\mu_0\right),
\]
affine invariance of the associated divergence restricts the nontrivial power branch to
\[
 U(z)=z^\gamma+c,
 \qquad
 V(z)=z^{1+\gamma},
\]
and the logarithmic branch to
\[
 U(z)=-\log z+c,
 \qquad
 V(z)=z,
\]
under the regularity assumptions stated in~\citet{KanamoriFujisawa2014Bernoulli}. On compact domains, the scale-invariant characterization extends to negative powers and includes the Itakura--Saito limit~\citep{Kanamori2014Entropy}. Combined with~\cref{thm:universality}, these results imply the following restricted characterization.

\begin{corollary}[Maximum-entropy families in the regular composite-score class]\label{cor:composite-characterization}
Under the regularity and invariance assumptions of the cited two-integral composite-score characterization, the induced projective maximum-entropy family is an exponential family in the logarithmic case and a $q$-exponential family in the power case.
\end{corollary}

\begin{remark}
\Cref{cor:composite-characterization} is not a characterization of all proper scoring rules. In particular, local scores depending on derivatives of the predictive density lie outside the two-integral class. Such scores have a separate characterization and, from order two onward, can often be evaluated without knowledge of the normalizing constant \citep{ParryDawidLauritzen2012}. Moreover, the negative-power results on noncompact Euclidean spaces used below are obtained directly from the projective power problem, not from the compact-domain scale-invariance theorem.
\end{remark}

\section{Universal optimizer under linear moment constraints}\label{sec:qexp}

\begin{theorem}[$q$-exponential representation]\label{thm:qexp}
Let $\gamma\in(-1,\infty)\setminus\{0\}$, and suppose that a maximum-entropy solution $p^*$ exists in $\calC_{\tau,\gamma}$ and satisfies the usual variational and Karush--Kuhn--Tucker regularity conditions. Then there exist $a\in\R$ and $\lambda\in\R^m$ such that
\begin{equation}
 p^*(x)=
 \begin{cases}
 \pos{a+\lambda^\top T(x)}^{1/\gamma},&\gamma>0,\\[1mm]
 \{a+\lambda^\top T(x)\}^{1/\gamma},&-1<\gamma<0,
 \end{cases}
\label{eq:qexp-general}
\end{equation}
where the bracket in the negative-power case must remain positive on the support of a finite interior solution. The constants are determined by normalization and the moment constraints.
\end{theorem}

\begin{proof}
By \cref{thm:universality}, it is sufficient to maximize $\calS_\gamma$. The Lagrangian is
\[
 \mathcal L(p)
 =
 \frac{1-\int p^{1+\gamma}\,\dd\mu_0}{\gamma}
 -\alpha\left(\int p\,\dd\mu_0-1\right)
 -\lambda^\top\left(\int Tp\,\dd\mu_0-\tau\right).
\]
On the region where $p>0$, the first variation gives
\[
 -\frac{1+\gamma}{\gamma}p(x)^\gamma
 -\alpha-\lambda^\top T(x)=0.
\]
After reparameterizing the multipliers, this becomes $p(x)^\gamma=a+\lambda^\top T(x)$. For $\gamma>0$, complementary slackness for the nonnegativity constraint produces the positive part in \cref{eq:qexp-general}. For $-1<\gamma<0$, the exponent is negative, so the affine bracket must be positive wherever the density is finite.
\end{proof}

\begin{corollary}[Shannon limit]\label{cor:shannon-limit}
Under a regular parameterization of the multipliers, the family in \cref{eq:qexp-general} converges as $\gamma\to0$ to the ordinary exponential-family form
\[
 p_0^*(x)\propto\exp\{\lambda^\top T(x)\}.
\]
\end{corollary}

\section{Mean and covariance constraints}\label{sec:mean-cov}

Let $\calX=\R^d$ with Lebesgue measure. Fix $\mu\in\R^d$ and a positive-definite matrix $V$, and write
\[
 Q_{\mu,V}(x)=(x-\mu)^\top V^{-1}(x-\mu).
\]
Let $\calP_{\mu,V}$ denote the class of probability densities with mean $\mu$ and covariance $V$ for which the relevant power integral is finite.

The covariance-constrained optimizer itself is classical in the R\'enyi-entropy literature. Generalized Gaussian densities are extremizers of sharp moment--entropy inequalities~\citep{LutwakYangZhang2005}, and the compactly supported Student-$r$ and heavy-tailed Student-$t$ branches are known to maximize R\'enyi entropy under covariance constraints~\citep{JohnsonVignat2007}. Extensions to bounds involving more than one moment constraint are also available~\citep{Reeves2020}. Accordingly, the contribution of~\cref{thm:mean-cov} is not a new optimizer formula. Its role is to place both branches inside the projective universality theorem, establish the exact parameterization by the prescribed covariance, and provide the support-radius identity used for the inverse calibration in \cref{sec:calibration}.

\begin{theorem}[Universal mean--covariance maximum entropy]\label{thm:mean-cov}
The following statements hold.

\medskip
\noindent
(i) If $\gamma>0$, the unique maximizer of every $\gamma$-admissible projective entropy over $\calP_{\mu,V}$ is
\begin{equation}
 p_\gamma(x;\mu,V)
 =
 \frac{A_{\gamma,d}}{|V|^{1/2}}
 \left(1-a_{\gamma,d}Q_{\mu,V}(x)\right)_+^{1/\gamma},
\label{eq:qgaussian-positive}
\end{equation}
where
\begin{align}
 a_{\gamma,d}
 &=\frac{\gamma}{2+\gamma(d+2)},
\label{eq:a-positive}\\
 A_{\gamma,d}
 &=
 \frac{a_{\gamma,d}^{d/2}
 \Gamma(1+\gamma^{-1}+d/2)}
 {\pi^{d/2}\Gamma(1+\gamma^{-1})}.
\label{eq:A-positive}
\end{align}
Its support satisfies
\begin{equation}
 \overline{\supp(p_\gamma)}
 =
 \left\{x:Q_{\mu,V}(x)\leq R_\gamma^2\right\},
 \qquad
 R_\gamma^2=d+2+\frac{2}{\gamma}.
\label{eq:Rgamma}
\end{equation}

\medskip
\noindent
(ii) If $-2/(d+2)<\gamma<0$, the unique maximizer is
\begin{equation}
 p_\gamma(x;\mu,V)
 =
 \frac{\Gamma((\nu+d)/2)}
 {\Gamma(\nu/2)\{\pi(\nu-2)\}^{d/2}|V|^{1/2}}
 \left(1+\frac{Q_{\mu,V}(x)}{\nu-2}\right)^{-(\nu+d)/2},
\label{eq:student-branch}
\end{equation}
where
\begin{equation}
 \nu=-\frac{2}{\gamma}-d>2.
\label{eq:nu-gamma}
\end{equation}
This is a multivariate Student distribution with degrees of freedom $\nu$, location $\mu$, scale matrix $(\nu-2)V/\nu$, and covariance $V$.
\end{theorem}

The complete proof, including the beta-integral calculations, covariance calibration, and global optimality arguments, is given in \cref{app:mean-cov-proof}. The proof is global rather than a purely formal Euler--Lagrange calculation: convexity or concavity of the power integral and the moment constraints certify optimality over the full feasible class.

\begin{corollary}[Gaussian limit]\label{cor:gaussian-limit}
As $\gamma\to0+$ or $\gamma\to0-$, the density $p_\gamma(x;\mu,V)$ converges pointwise, and in distribution, to the Gaussian density $N(\mu,V)$.
\end{corollary}

\section{Acceptance-region calibration}\label{sec:calibration}

For positive $\gamma$, \cref{eq:Rgamma} gives the squared Mahalanobis radius of the support. Reversing this relation turns a prescribed acceptance region into a deformation parameter.

\begin{theorem}[Universal acceptance-region calibration]\label{thm:calibration}
Let $d\geq1$, let $R^2>d+2$, and define
\begin{equation}
 \gamma_R=\frac{2}{R^2-d-2},
 \qquad
 \alpha_R=\frac{1}{\gamma_R}=\frac{R^2-d-2}{2}.
\label{eq:gammaR}
\end{equation}
For a prescribed mean $\mu$ and positive-definite covariance $V$, let
\[
 A_{\mu,V}(R)
 =\left\{x\in\R^d:Q_{\mu,V}(x)\leq R^2\right\}.
\]
Then the unique maximizer of every $\gamma_R$-admissible projective entropy over all densities with mean $\mu$ and covariance $V$ is
\begin{equation}
 q_R(x;\mu,V)
 =
 \frac{C_{d,R}}{|V|^{1/2}}
 \left(1-\frac{Q_{\mu,V}(x)}{R^2}\right)_+^{\alpha_R},
\label{eq:qR}
\end{equation}
where
\begin{equation}
 C_{d,R}
 =
 \frac{\Gamma(R^2/2)}
 {(\pi R^2)^{d/2}\Gamma((R^2-d)/2)}.
\label{eq:CdR}
\end{equation}
Moreover,
\[
 \overline{\supp(q_R)}=A_{\mu,V}(R).
\]
Thus the acceptance ellipsoid is generated endogenously by the calibrated maximum-entropy problem rather than imposed as an additional support constraint.
\end{theorem}

\begin{proof}
Equation \cref{eq:gammaR} is equivalent to
\[
 R^2=d+2+\frac{2}{\gamma_R}.
\]
Substitution into \cref{eq:a-positive} gives $a_{\gamma_R,d}=R^{-2}$, while $1/\gamma_R=\alpha_R$. Equation \cref{eq:qR} therefore follows from \cref{thm:mean-cov}(i). The transformation of the normalizing constant from \cref{eq:A-positive} to \cref{eq:CdR} is verified in \cref{app:calibration-constant}. Since \cref{thm:mean-cov} imposes only mean and covariance constraints, no external support restriction is used. Universality follows from \cref{thm:universality}.
\end{proof}

\begin{corollary}[Affine equivariance]\label{cor:affine-equivariance}
Under an invertible affine transformation $y=b+Bx$, the push-forward of $q_R(x;\mu,V)$ is $q_R(y;b+B\mu,BVB^\top)$. Hence the calibrated reference family is affine equivariant.
\end{corollary}

\begin{remark}[Statistical construction of the reference distribution]\label{rem:robust-inputs}
The theorem separates the construction into interpretable inputs. The center and shape matrix may be supplied by classical or robust location and scatter estimators. The radius $R$ may come from an externally specified admissible region, a safety requirement, or a screening rule. The theorem then determines both the deformation parameter and the full reference density. Robustness is not automatic: if nonrobust estimates of $\mu$ and $V$ are contaminated, the calibrated reference distribution inherits that contamination.
\end{remark}

\begin{remark}[Limiting cases]\label{rem:limits}
As $R^2\downarrow d+2$, one has $\gamma_R\to\infty$ and $\alpha_R\downarrow0$, and $q_R$ approaches the uniform density on the ellipsoid of squared radius $d+2$. As $R^2\to\infty$, one has $\gamma_R\to0$, and the Gaussian maximum-entropy density is recovered. The radius therefore interpolates continuously between a hard bounded reference and the Gaussian reference.
\end{remark}

\begin{remark}[Matching a Gaussian probability ellipsoid]\label{rem:chi-square}
One may set $R^2$ equal to a $\chi_d^2$ quantile when the desired boundary is inherited from a Gaussian probability ellipsoid. A finite positive $\gamma_R$ exists only if that quantile exceeds $d+2$. This restriction is not technical: an ellipsoid narrower than this threshold cannot simultaneously be the support of the family in \cref{eq:qR} and preserve covariance $V$.
\end{remark}

\section{Projective shape and radial scale}\label{sec:shape-scale}

Write an unnormalized nonnegative function as
\[
 r=mp,
 \qquad
 m=\int r\,\dd\mu_0,
 \qquad
 \int p\,\dd\mu_0=1.
\]
The projective power functional depends only on the shape:
\[
 \calJ_\gamma([r])=\int p^{1+\gamma}\,\dd\mu_0.
\]
Whether a statistical score identifies the radial scale $m$ is a separate issue. Substituting $g=mp$ into the H\"older composite score \cref{eq:holder-score} gives
\begin{equation}
 S_{\gamma,\phi}(f,mp)
 =
 m^{1+\gamma}\left(\int p^{1+\gamma}\,\dd\mu_0\right)
 \phi\!\left(
 \frac{\int fp^\gamma\,\dd\mu_0}
 {m\int p^{1+\gamma}\,\dd\mu_0}
 \right).
\label{eq:radial-holder}
\end{equation}

The following elementary consequence of the existing Hölder-score representation clarifies which member of the class discards the radial scale.
\begin{lemma}[Radial invariance of H\"older scores]
\label{lem:radial}
As a direct consequence of the H\"older-score representation~\citep{KanamoriFujisawa2014Bernoulli}, a H\"older composite score is
invariant under rescaling of its second argument,
\[
S_{\gamma,\phi}(f,mg)=S_{\gamma,\phi}(f,g),
\qquad m>0,
\]
if and only if $\phi(z)=-z^{1+\gamma}$, that is, if and only if it
belongs to the pseudo-spherical equivalence class. The scale
invariance of the pseudo-spherical score in unnormalized models was
used explicitly by~\citet{KanamoriFujisawa2015}.
\end{lemma}

\begin{proof}
Let $t=(\int fp^\gamma\,\dd\mu_0)/(\int p^{1+\gamma}\,\dd\mu_0)>0$. Radial invariance in \cref{eq:radial-holder} implies
\[
 m^{1+\gamma}\phi(t/m)=\phi(t).
\]
Putting $z=t/m$ shows that $\phi(z)/z^{1+\gamma}$ is constant. Hence $\phi(z)=Cz^{1+\gamma}$, and the normalization $\phi(1)=-1$ gives $C=-1$. Direct substitution proves the converse.
\end{proof}

This lemma separates two methodological questions. Projective maximum entropy selects a normalized shape, and that shape is universal across the admissible entropy class. By contrast, scale identification depends on the chosen score. Pseudo-spherical and logarithmic gamma scores eliminate the radial scale, whereas density-power scores retain it. This distinction is exploited by unnormalized-model procedures that estimate a contamination proportion jointly with shape parameters~\citep{KanamoriFujisawa2015} and by localized homogeneous divergences for discrete unnormalized models~\citep{TakenouchiKanamori2017}. Homogeneous scoring rules have also been used to remove arbitrary scaling constants in model comparison~\citep{DawidMusio2015}. The projective principle therefore does not replace scale-estimation methodology; it specifies the shape component to which such methods may be attached.

The same limitation applies to energy-based models. Although $\calJ_\gamma([r_\theta])$ is invariant to additive constants in the energy, it still contains the ratio
\[
 \calJ_\gamma([r_\theta])
 =
 \frac{\int r_\theta^{1+\gamma}\,\dd\mu_0}
 {\left(\int r_\theta\,\dd\mu_0\right)^{1+\gamma}}.
\]
Thus the variational principle does not by itself yield a partition-function-free learning algorithm. Such algorithms require additional estimating structure, as in score matching \citep{Hyvarinen2005}, statistically efficient estimation frameworks for unnormalized models \citep{UeharaKanamoriTakenouchiMatsuda2020}, pseudo-spherical contrastive objectives \citep{YuSongSongErmon2021}, and recent unifications of unnormalized-density learning through noise-contrastive estimation \citep{RyuShahWornell2025}. The projective maximum-entropy result is complementary to these methods: it characterizes the target shape selected by a variational principle, but does not prescribe a computational estimator for an arbitrary energy-based model.

\section{Numerical illustrations}\label{sec:illustrations}

The figures are analytical illustrations rather than empirical validation. They show how the acceptance radius controls the reference density and how the calibrated family connects bounded and Gaussian limits.

Figure~\ref{fig:density-family} displays the one-dimensional family for several values of $R^2$. As $R^2$ approaches the minimum value $3$, the density approaches a uniform law on a bounded interval. As $R^2$ increases, the density approaches the standard Gaussian.

\begin{figure}[tbp]
\centering
\includegraphics[width=0.84\linewidth]{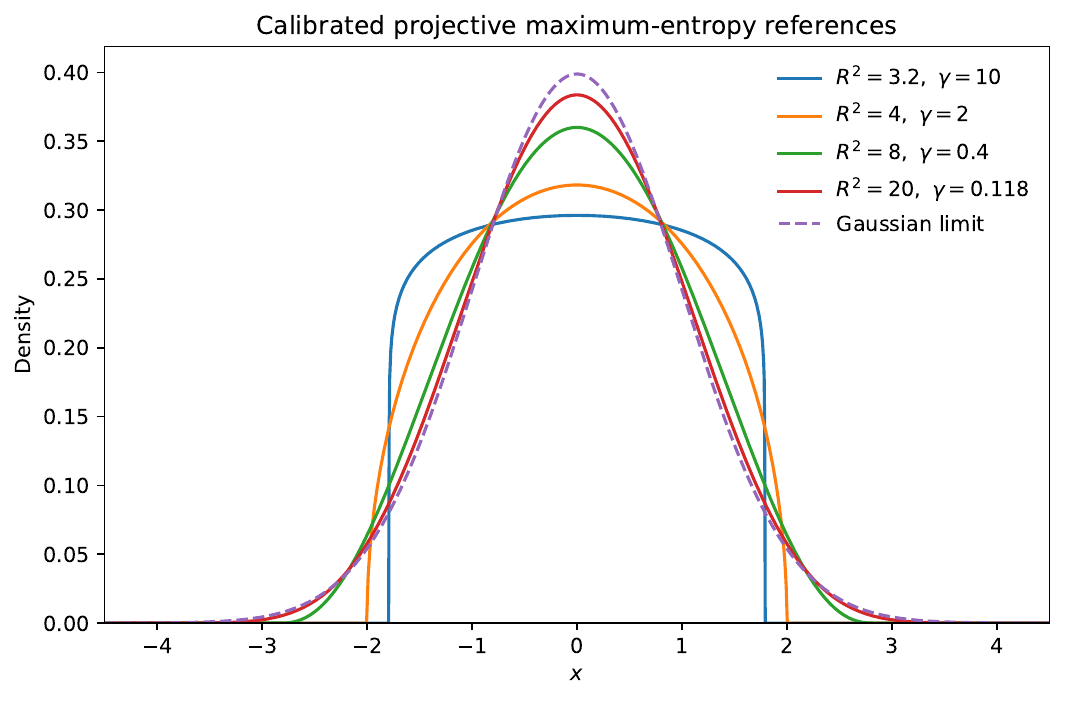}
\caption{Calibrated projective maximum-entropy references in one dimension with $\mu=0$ and $V=1$. Each deformation parameter is determined by $\gamma_R=2/(R^2-3)$.}
\label{fig:density-family}
\end{figure}

Figure~\ref{fig:calibration-curve} shows the relation between $R^2$ and $\gamma_R$ in several dimensions. The parameter increases rapidly near the minimum feasible radius $d+2$ and approaches the Gaussian limit $\gamma_R=0$ for large radii.

\begin{figure}[tbp]
\centering
\includegraphics[width=0.80\linewidth]{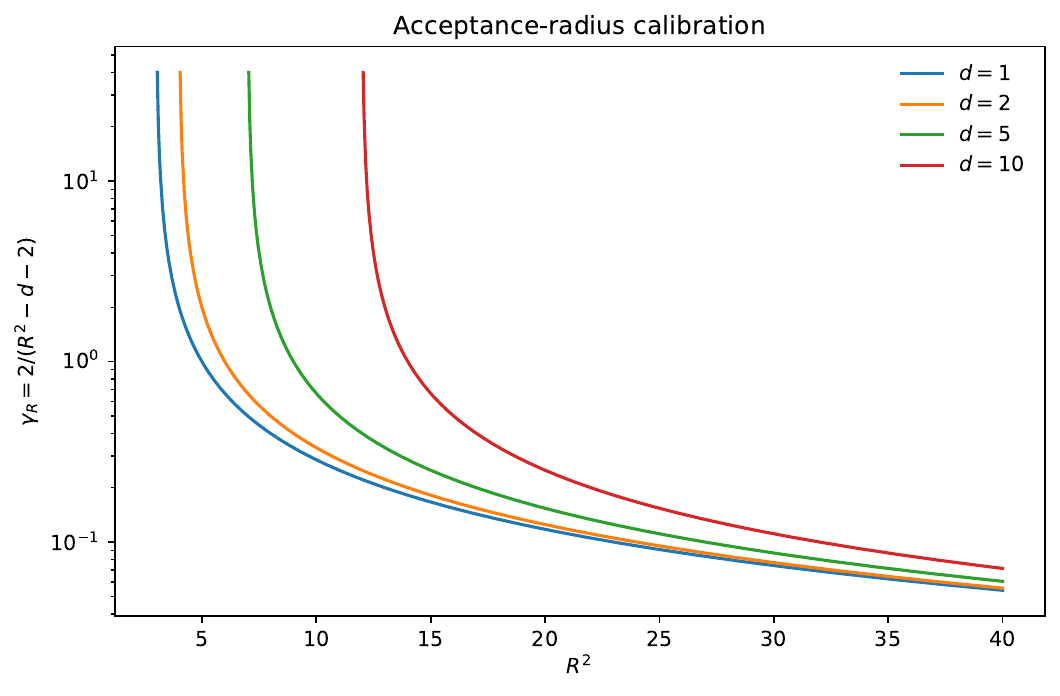}
\caption{Acceptance-radius calibration curves. The relation is $\gamma_R=2/(R^2-d-2)$ for $R^2>d+2$.}
\label{fig:calibration-curve}
\end{figure}

For $Q<R^2$, the relative negative log-density of the bounded reference is
\begin{equation}
 \ell_R(Q)
 =-\alpha_R\log\left(1-\frac{Q}{R^2}\right).
\label{eq:boundary-profile}
\end{equation}
Unlike the Gaussian profile $Q/2$, this quantity diverges at the prescribed boundary. 
Figure~\ref{fig:boundary-profile} illustrates how the acceptance boundary is encoded directly in the reference density.

\begin{figure}[tbp]
\centering
\includegraphics[width=0.80\linewidth]{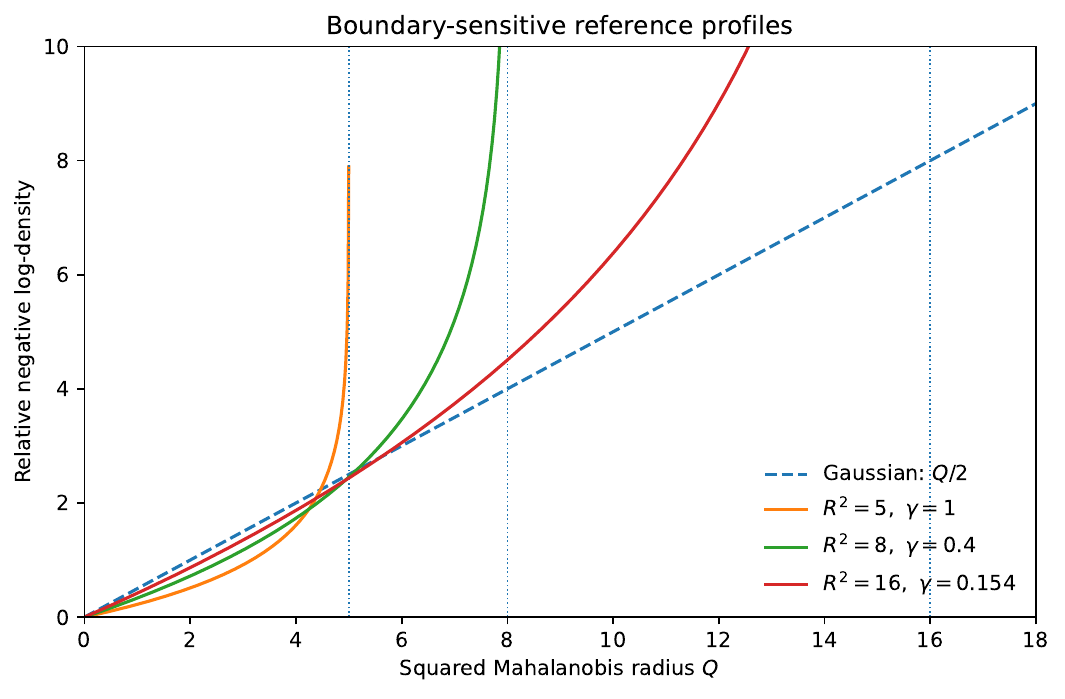}
\caption{Relative negative log-density as a function of squared Mahalanobis radius $Q$. Vertical dotted lines indicate the calibrated support boundaries.}
\label{fig:boundary-profile}
\end{figure}

\section{Discussion}\label{sec:discussion}

The distributions in \cref{thm:mean-cov} are established members of the power-entropy maximum-entropy family. The contribution of this paper is therefore not a new $q$-Gaussian or Student family. It is a statistical characterization with three linked components: projective formulation for unnormalized shapes, a universality theorem across a broad entropy class, and inverse calibration from an acceptance region to a unique affine-equivariant reference density.

The universality result also clarifies what is and is not shared by different scoring rules. Their diagonal entropy order may be identical even when their off-diagonal divergences, influence functions, regression properties, and scale-identification behavior differ. Maximum entropy sees the diagonal ordering; statistical estimation may depend critically on the remaining structure. This is why the projective shape principle and the radial-scale proposition should be viewed as complementary rather than interchangeable results.

The calibrated radius gives a direct interpretation of positive deformation. It controls the support geometry through
\[
 R_\gamma^2=d+2+\frac{2}{\gamma}.
\]
For negative deformation in the finite-covariance range, the analogous interpretation is tail thickness through
\[
 \nu=-\frac{2}{\gamma}-d.
\]
Thus the sign of $\gamma$ distinguishes bounded-support and heavy-tailed reference design. The appropriate branch depends on the statistical purpose. A bounded reference is unsuitable when extreme-tail risk is itself the target of inference, while a heavy-tailed reference may be undesirable when the domain has a genuine physical boundary.

Several extensions remain open. First, moment constraints could be replaced by robust estimating equations or escort moments; universality would then need to be re-examined because the feasible set itself may depend on the deformation. Second, the affine-invariance characterization could be extended beyond two-integral composite scores to derivative-based local scores. Third, for energy-based models, one needs conditions under which the projective power functional or its gradient can be estimated without evaluating the normalizing constant.

\section{Conclusion}\label{sec:conclusion}

We formulated maximum entropy on the projective space of nonnegative measures and identified the normalized power integral as the common variational object behind a broad class of entropy and scoring-rule constructions. Any admissible monotone transform of this functional yields the same optimizer under fixed moment constraints. Under general linear constraints, the optimizer is $q$-exponential. Under mean and covariance constraints, it is a compactly supported $q$-Gaussian for positive deformation and a Student-type density for negative deformation.

The acceptance-region calibration theorem provides the main statistical design result. A Mahalanobis ellipsoid with squared radius $R^2>d+2$ uniquely determines $\gamma_R=2/(R^2-d-2)$ and hence a unique affine-equivariant maximum-entropy reference density. Its support is exactly the prescribed ellipsoid, although no support constraint is imposed in the optimization problem. The result provides a principled route from estimated location and scatter and an externally meaningful admissible region to a fully specified bounded-support reference distribution.

\section*{Acknowledgments}
I would like to express my sincere gratitude to Professors Akifumi Okuno and Hidetoshi Shimodaira for providing the inspiration for this research through our discussions.
Part of this work is supported by JSPS KAKENHI No.~JP26K02989, JP25H01494, and JP26K23861.

\newpage
\appendix

\section{Proof details for the mean--covariance theorem}\label{app:mean-cov-proof}

By the affine change of variables $z=V^{-1/2}(x-\mu)$, it is sufficient to prove \cref{thm:mean-cov} for $\mu=0$ and $V=I_d$.

\subsection{Positive deformation: normalization and covariance}\label{app:positive-integrals}

Let $m=1/\gamma$ and consider
\[
 g(z)=A(1-a\norm{z}^2)_+^m.
\]
Using polar coordinates and the beta integral,
\begin{align}
 \int_{\R^d}(1-a\norm{z}^2)_+^m\,\dd z
 &=
 \frac{2\pi^{d/2}}{\Gamma(d/2)}
 \int_0^{1/\sqrt a}r^{d-1}(1-ar^2)^m\,\dd r \\
 &=
 \pi^{d/2}a^{-d/2}
 \frac{\Gamma(m+1)}{\Gamma(m+1+d/2)}.
\label{eq:positive-beta-normalization}
\end{align}
Hence
\[
 A=a^{d/2}
 \frac{\Gamma(m+1+d/2)}{\pi^{d/2}\Gamma(m+1)}.
\]
Spherical symmetry gives $\E_g[Z]=0$. A second beta-integral calculation yields
\begin{equation}
 \E_g[\norm{Z}^2]
 =
 \frac{d}{2a(m+1+d/2)}.
\label{eq:positive-second-moment}
\end{equation}
Setting this equal to $d$ gives
\[
 a=\frac{1}{2(m+1+d/2)}
 =\frac{\gamma}{2+\gamma(d+2)},
\]
which proves \cref{eq:a-positive}--\cref{eq:A-positive} and $\Cov_g(Z)=I_d$.

\subsection{Positive deformation: global optimality}\label{app:positive-optimality}

Let $q$ be any density with mean zero and covariance $I_d$. Since $t\mapsto t^{1+\gamma}$ is strictly convex,
\begin{equation}
 \int q^{1+\gamma}\,\dd z
 \geq
 \int g^{1+\gamma}\,\dd z
 +(1+\gamma)\int g^\gamma(q-g)\,\dd z.
\label{eq:positive-convexity}
\end{equation}
Writing $c=A^\gamma$, one has
\[
 g^\gamma=c(1-a\norm{z}^2)_+.
\]
Because $(1-u)_+\geq1-u$ and $q\geq0$,
\[
 \int g^\gamma q\,\dd z
 \geq
 c\int(1-a\norm{z}^2)q\,\dd z
 =c(1-ad).
\]
The density $g$ is supported where the positive part is active, so
\[
 \int g^\gamma g\,\dd z=c(1-ad).
\]
The linear term in \cref{eq:positive-convexity} is therefore nonnegative, and $g$ globally minimizes the power integral. Strict convexity gives uniqueness. By \cref{thm:universality}, $g$ maximizes every admissible entropy.

\subsection{Negative deformation: normalization, covariance, and optimality}\label{app:negative-proof}

Let $-2/(d+2)<\gamma<0$ and define
\[
 \nu=-\frac{2}{\gamma}-d>2,
 \qquad
 b=\frac{1}{\nu-2}
 =-\frac{\gamma}{2+\gamma(d+2)}>0.
\]
The candidate density is
\[
 g(z)=B(1+b\norm{z}^2)^{1/\gamma}.
\]
Since $1/\gamma=-(\nu+d)/2$, the standard multivariate beta integral gives
\[
 B=
 \frac{\Gamma((\nu+d)/2)}
 {\Gamma(\nu/2)\{\pi(\nu-2)\}^{d/2}}.
\]
This is the multivariate Student density with degrees of freedom $\nu$ and scale matrix $(\nu-2)I_d/\nu$, so $\E_g[Z]=0$ and $\Cov_g(Z)=I_d$.

For any competing density $q$ with the same first two moments, strict concavity of $t\mapsto t^{1+\gamma}$ yields
\begin{equation}
 \int q^{1+\gamma}\,\dd z
 \leq
 \int g^{1+\gamma}\,\dd z
 +(1+\gamma)\int g^\gamma(q-g)\,\dd z.
\label{eq:negative-concavity}
\end{equation}
Here
\[
 g^\gamma=B^\gamma(1+b\norm{z}^2),
\]
which is an affine combination of $1$ and $\norm{z}^2$. Equal mass and equal second moments therefore make the linear term in \cref{eq:negative-concavity} exactly zero. Hence $g$ globally maximizes the power integral, and strict concavity gives uniqueness. This proves the negative branch of \cref{thm:mean-cov}.

\subsection{Calibration constant}\label{app:calibration-constant}

When $R^2=d+2+2/\gamma$,
\[
 \frac{1}{\gamma}=\frac{R^2-d-2}{2}=\alpha_R,
 \qquad
 a_{\gamma,d}=\frac{1}{R^2}.
\]
Substituting these identities into \cref{eq:A-positive} gives
\begin{align*}
 A_{\gamma,d}
 &=
 R^{-d}
 \frac{\Gamma(1+\alpha_R+d/2)}
 {\pi^{d/2}\Gamma(1+\alpha_R)}\\
 &=
 \frac{\Gamma(R^2/2)}
 {(\pi R^2)^{d/2}\Gamma((R^2-d)/2)}
 =C_{d,R},
\end{align*}
which proves \cref{eq:CdR}.


\end{document}